\documentclass[pdflatex,sn-mathphys]{sn-jnl}
\usepackage{psfrag}
\jyear{2023}%

\newcommand{\norm}[1]{\left\lVert#1\right\rVert}
\newcommand{\inner}[2]{\langle#1,#2\rangle}

\newcommand{\eu}{\mathbf{e}}

\theoremstyle{thmstyleone}%
\newtheorem{theorem}{Theorem}%  meant for continuous numbers
\newtheorem{corollary}{Corollary}[theorem]
\newtheorem{lemma}[theorem]{Lemma}
\newtheorem{proposition}[theorem]{Proposition}% 

\theoremstyle{thmstyletwo}%
\newtheorem{remark}{Remark}%

\theoremstyle{thmstylethree}%
\newtheorem{definition}{Definition}%

\begin{document}
\nocite{4321}
\title[Stability and long-time behaviour of a rigid body containing a damper]{Stability and long-time behaviour of a rigid body containing a damper}

\author[1]{\fnm{Evan} \sur{Arsenault}}\email{evan.arsenault@mail.utoronto.ca}
\equalcont{This authors contributed equally to this work.}
\affil[1]{\orgdiv{Edward S. Rogers Sr. Department of Electrical and Computer Engineering}, \orgname{University of Toronto}, \orgaddress{\street{10 King's College Road}, \city{Toronto}, \postcode{M5S 3G4}, \state{Ontario}, \country{Canada}}}

\author*[2]{\fnm{Giusy} \sur{Mazzone}}\email{giusy.mazzone@queensu.ca}
\affil[2]{\orgdiv{Department of Mathematics and Statistics}, \orgname{Queen's University}, \orgaddress{\street{48 University Avenue}, \city{Kingston}, \postcode{K7L 3N6}, \state{Ontario}, \country{Canada}}}

\abstract{
This paper provides a comprehensive analysis of stability and long-time behaviour of a coupled system constituted by two rigid bodies separated by a thin layer of lubricant. We show that permanent rotations of the whole system, with the solids at relative rest, are exponentially stable if and only if the axis of rotation is the principal axis of inertia corresponding to the largest moment of inertia of the outer body. All other equilibria are normally hyperbolic, and hence unstable. In addition, we show that all solutions to the governing equations converge to an equilibrium configuration, no matter which initial conditions are chosen. Numerical evidence of the above results as well as conditions ensuring attainability of the stable configurations are also presented. 

For the stability analysis we use a linearization principle for dynamical systems possessing a center manifold. The characterization of the long-time behaviour is obtained by a careful analysis of the partially dissipative system of equations governing the motion of the coupled system. 
}

\keywords{rigid body dynamics, free rotation, partially dissipative system, normally stable equilibrium, normally hyperbolic equilibrium, long-time behaviour, stabilization}

\pacs[MSC Classification]{70K25,70E17,70E50,34A34,34D05,34D20}

\maketitle
\subsection*{Declarations}
\subsubsection*{Funding}

Evan Arsenault gratefully acknowledges the support of the Natural Sciences and Engineering Research Council of Canada (NSERC) through the NSERC Undergraduate Student Research Award ``Long-time dynamics of rigid bodies with a damper''.

Giusy Mazzone gratefully acknowledges the support of the Natural Sciences and Engineering Research Council of Canada (NSERC) through the NSERC Discovery Grant ``Partially dissipative systems with applications to fluid-solid interaction problems''.

\subsubsection*{Conflicts of interest}
The authors have no relevant financial or non-financial interests to disclose.
\subsubsection*{Availability of data and material}
Not applicable. 

\subsubsection*{Code availability}
Not applicable. 

\subsubsection*{Author's contributions}
All authors contributed to the investigations here reported. The first draft of the manuscript was written by Evan Arsenault and all authors commented on previous versions of the manuscript. All authors read and approved the final manuscript.

\subsection*{Acknowledgements}
Not applicable. 

\pagebreak
\section{Introduction}

Consider a system constituted by two rigid bodies $\mathcal B_1$ and $\mathcal B_2$, with $\mathcal B_2$ (strictly) contained in a spherical cavity within $\mathcal B_1$. The solids are separated only by a thin  spherical layer of lubricant\footnote{The lubricant is a viscous incompressible fluid.} that completely fills the gap between the two solids (see Figure \ref{fig:damped_solid}). We further assume that the inner solid $\mathcal B_2$ is a homogeneous spherical rigid body with radius $a\gg h$, where $h$ denotes the (constant) thickness of the gap between $\mathcal B_2$ and the surface of the cavity within $\mathcal B_1$.
\begin{figure}[h]
\begin{center}
\psfrag{1}{$\mathcal B_1$}
\psfrag{2}{$\mathcal B_2$}
\psfrag{G}{$G$}
\psfrag{a}{$a$}
\psfrag{h}{\footnotesize $h$}
\includegraphics[width=.55\textwidth]{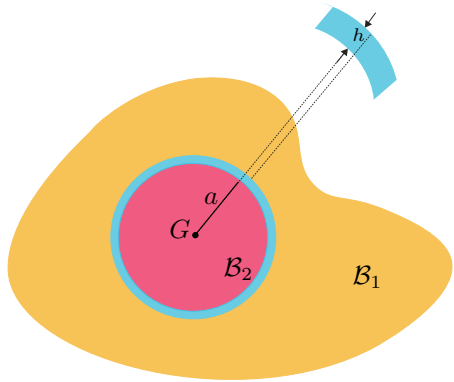}
\caption{A rigid body with a damper}\label{fig:damped_solid}
\end{center}
\end{figure}

In this paper, we consider the {\em free rotations} of the whole system about the center of mass $G$ of $\mathcal B_1$, i.e., we assume that there are no external forces acting on the system, and $G$ is a fixed point (see equation \eqref{outerangmom}). In addition, we ignore any translational motion of $\mathcal B_2$ within the cavity, and assume that the center\footnote{The geometrical center $C$ of $\mathcal B_2$ is also its center of mass since we assumed that $\mathcal B_2$ is a homogeneous rigid body.} $C$ of $\mathcal B_2$ coincides with $G$\footnote{Following \cite{evo}, this latter assumption can be replaced by assuming that the entire mass of $\mathcal B_2$ is concentrated at $C$.}. We further suppose that the surface force (which is an internal force, due to the presence of the lubricant) between the two solids produces a torque which is proportional to the velocity of $\mathcal B_2$ relative to $\mathcal B_1$ (see equation \eqref{innerangmom}). 

It is worth noticing that if there were no lubricant between the solids, then the motion of the two rigid bodies would be completely uncoupled (and thus, uninteresting). Furthermore, if there were no inner solid within the cavity, then the generic motion of the outer rigid body $\mathcal B_1$ would be a motion \`a la Poinsot. For example, depending on the initial conditions, $\mathcal B_1$ could perform a regular precession. Objective of this work is to show that the inner solid $\mathcal B_2$, subject to the type of torque described above (due to the lubricant), would {\em stabilize} the motion of the outer solid and take the whole system to a state (an equilibrium configuration, in fact, see Theorem \ref{equilibria}) that is the relative rest of the solids with the whole system spinning with constant angular velocity around one of the principal axes of inertia of $\mathcal B_1$. In this sense (also inspired by \cite{chern}), we refer to $\mathcal B_2$ as the ``damper'', and to the whole system as a ``rigid body with a spherical damper''. 

According to  \cite{chern} (see also \cite[Section 7.4]{evo} and references therein), M.~ A. Lavrentyev considered this problems as a lumped mass parameter system approximating the motion of a rigid body with a cavity filled with a viscous fluid. 
Chernous'ko (see, e.g., \cite{chern,evo}), obtained conditions on physical quantities (like fluid density and viscosity) under which equations \eqref{ODE}, describing the motions of a rigid body with a spherical damper, would as well describe the motion of a rigid body with a spherical cavity filled with an incompressible viscous fluid (at low Reynolds number). The motion of fluid-filled rigid bodies is studied in connection with geophysical applications in the search of sources for the geomagnetic field, in structural engineering for the design of tuned liquid dampers, as well as in space engineering for the stabilization of spacecrafts (see \cite{giu} and references therein).  

In \cite{chern,evo}, the authors derived necessary and sufficient conditions for the stability of permanent rotations\footnote{ Permament rotations are rigid body rotations occurring with constant angular velocity. } of the  system, as a whole rigid body, around one of the principal axis of inertia of the outer solid. Let us be more precise. Consider the equilibrium configuration which is characterized by $\mathcal B_2$ at relative rest with respect to $\mathcal B_1$, and the whole system of rigid body with a damper rotating with constant angular velocity $\Omega^*$ around the principal axes of inertia $\mathsf a$ of $\mathcal B_1$. Let $A$ be the corresponding principal moment of inertia of $\mathcal B_1$, and denote with $B$ and $C$ the remaining moments of inertia. Using a spectral stability argument, Chernous'ko et al. show that a necessary condition for the stability (in the sense of Lyapunov) of $\Omega^*$ is that $A\ge B, C$. In the same works, sufficient conditions are obtained using the classical Lyapunov method. It is then found that $\Omega^*$ is stable if $A>B,C$. These results are nonlinear stability results which are silent about the asymptotic stability properties of equilibria as well as the long-time behaviour of generic motions. For example, it is not known whether trajectories starting sufficiently close to $\Omega^*$ would eventually (as time goes to infinity) converge to an equilibrium, and if so, with what rate. 

In this paper, we show that the condition $A\ge B,C$ is a necessary and sufficient condition for the stability of $\Omega^*$ (see Theorem \ref{th:nonlinear_stability}), and we also prove that {\em all  trajectories corresponding to solutions of the equations of motion will converge to an equilibrium configuration at an exponential rate, no matter how one chooses the initial conditions} (see Theorem \ref{th:long-time}).  The stability results are obtained using a {\em principle of linearized stability} for autonomous systems of ordinary differential equations with normally stable and/or normally hyperbolic equilibria (see Definition \ref{def_norhyp}, and Theorems \ref{normalstable} \& \ref{normalhyper}).  A spectral analysis of the equations of motion \eqref{ODE} shows that equilibria are, indeed, either normally stable or normally hyperbolic (see Theorem \ref{pruss-theorem}). The long-time behaviour of generic  motions $z(t;z_0)$\footnote{ $z(t;z_0)=(\Omega(t),\Omega_1(t))^T$ is the (unique) solution to \eqref{ODE} corresponding to the initial data $z_0=(\Omega_0,\Omega_{10})^T\in \mathbb{R}^3\times\mathbb{R}^3$. } is obtained by showing that the corresponding $\omega$-limit set $\omega(z_0)$ is non-empty, compact and connected, it is made of either only normally stable equilibria or only normally hyperbolic ones, and 
\[
\lim_{t \rightarrow \infty} \text{d}(z(t;z_0),\omega(z_0))=0. 
\]
Convergence to an equilibrium point is then a consequence of the asymptotic stability properties of normally stable and normally hyperbolic equilibria according to  Theorems \ref{normalstable} \& \ref{normalhyper}. 

This paper also provides sufficient conditions on the attainability of the stable equilibrium configurations. In addition, we present some results concerning the numerical approximation of solutions to \eqref{ODE}. These results provide further insights on the stability conditions as well as on the attainability of equilibria. 

Beside their novelty, the results of this paper are non-trivial considering that the equations of motion are only {\em partially dissipative}. In fact, even if the total kinetic energy of the whole system decreases along the trajectories, the two solids must move while keeping the total angular momentum conserved at all times (see Proposition \ref{prop:balances}). In particular, neither the balance of kinetic energy \eqref{tderiv-kin}, nor the Lyapunov function constructed in \cite{chern,evo} could directly\footnote{Say, through a Gronwall-type lemma. } provide a decay with rate for the relative velocity of the solids. We show that an exponential decay of the solids relative velocity happens after ``some time'' the system has been moving, see Theorem \ref{th:long-time} and the numerical results in Figures \ref{fig-unstable} \& \ref{fig-attain}. Another point of interest is the way we prove the spectral stability properties of our equilibria which differs from that in  \cite{chern,evo}, and it has a more geometrical flavour as it shows how the spectrum of the linearization moves while changing the ordering of the moment of inertia of $\mathcal B_1$ (see proof of Theorem \ref{pruss-theorem}). We believe that such a complete and rigorous analytical treatment could be beneficial in the analysis of similar problems in rigid body dynamics involving the design of dampening mechanisms. 

Finally, this paper is in line with the recent work of the second author and collaborators concerning the stabilization of rigid bodies (\cite{DiGaMaZu,MaPrSi,MaPrSi19,giu2}). In particular, in \cite{giu2}, the second author of this paper has investigated the existence of solutions, and provided a preliminary analysis of the long-time behaviour of generic trajectories for the (more complete) fluid-solid interaction problem of rigid bodies separated by a gap filled by a viscous incompressible fluid (obeying the Navier-Stokes equations). In this respect, the current work could provide some insights for the study of stability and long-time behaviour of the physical system considered in \cite{giu2}. 

Here is the plan of the paper. We begin with Section \ref{sec:notation} containing some notation and the general stability theorems for autonomous systems that we will use for our physical problem. In Section \ref{sec:EOM}, we derive the equations governing the motion of a rigid body with a damper, and we provide preliminary properties of their solutions (like global existence, uniqueness, and energy balances). In Section \ref{sec:equilibria}, we provide a linear and nonlinear stability analysis of the equilibria. The long-time behaviour of generic trajectories is completely characterized in Section \ref{sec:long-time}. Numerical evidence of the above results is provided in Section \ref{sec:numerics}. 

\section{Notation and useful results}\label{sec:notation}

We use the symbol $\norm{\cdot}$ for the Euclidean norm in the $d$-dimensional Euclidean space $\mathbb{R}^d$, and denote with $\inner{\cdot}{\cdot}$ the standard dot product in $\mathbb{R}^d$. We recall that $\norm{u}=\sqrt{\inner{u}{u}}$ for every $u\in \mathbb{R}^d$.  In addition, for $r>0$ and $x_0\in \mathbb{R}^d$, $B_r(x_0)$ denotes the open ball centered at $x_0$ and with radius $r$, i.e., 
\[
B_r(x_0):=\{x\in \mathbb{R}^d:\; \norm{x-x_0}<r\}.
\]
Similarly,  
\[
\overline{B_r}(x_0):=\{x\in \mathbb{R}^d:\; \norm{x-x_0}\le r\}
\]
denotes the closed ball centered at $x_0$ and with radius $r$. 

In $\mathbb{C}^d$, we consider the inner product (over the field $\mathbb{C}$)\begin{multline*}
\inner{z}{w}_{\mathbb C^d}:=\inner{z}{\overline w}%=\inner{\mathsf{Re}(z)+i\;\mathsf{Im}(z)}{\mathsf{Re}(w)-i\;\mathsf{Im}(w)}
\\
\quad=\inner{\mathsf{Re}(z)}{\mathsf{Re}(w)}+\inner{\mathsf{Im}(z)}{\mathsf{Im}(w)}+i\left(\inner{\mathsf{Im}(z)}{\mathsf{Re}(w)}-\inner{\mathsf{Re}(z)}{\mathsf{Im}(w)}\right),
\end{multline*}
and associated norm $\norm{z}_{\mathbb C^d}=\sqrt{\inner{z}{z}_{\mathbb C^d}}=\sqrt{(\mathsf{Re}(z))^2+(\mathsf{Im}(z))^2}$. In the above, we have denoted with $\mathsf{Re }(\cdot)$  and $\mathsf{Im }(\cdot)$ the real and imaginary part of a complex number, respectively. 

For a matrix $A\in \mathbb{R}^{d\times d}$, $\sigma(A)$ denotes the spectrum of $A$, i.e., 
\[
\sigma(A)=\{\lambda\in \mathbb{C}:\; \text{there exists }u\in \mathbb{R}^d\setminus\{0\}\text{ such that }Au=\lambda u\}.
\]
$A^T$ denotes the transpose of $A$. $\mathsf{SO}(3)$ denotes the special orthogonal group of all orthogonal matrices with determinant 1. We recall that for a skew-symmetric matrix $A\in \mathbb{R}^{d\times d}$ (i.e., $A^T=-A$), there exists a vector $\alpha\in \mathbb{R}^d$ (called {\em axial vector}) such that $Av=\alpha\times v$ for every $v\in \mathbb{R}^d$. 

This paper makes use of the following useful results, the first being a modified Gronwall lemma. 
\begin{lemma}
\label{gronwall}
Suppose that a function $y\in L^{\infty}(0,\infty), y \geq 0$, satisfies the following inequality for all $t \geq 0$:
\begin{equation*}
    \dot y(t) \leq -Cy(t) + F(t),
\end{equation*}
where $C > 0$ is a constant, and $F \in L^q(a,\infty)\cap L^1_{\text{loc}}(0,\infty)$, for some $a > 0$ and $q \in [1, \infty)$, satisfies $F(t) \geq 0$ for a. a. $t \geq 0$. Then:
\begin{equation*}
    \lim_{t \rightarrow \infty} y(t) = 0
\end{equation*}

If $F \equiv 0$, then:
\begin{equation*}
    y(t) \leq y(0)e^{-Ct},\quad \text{for all }t \geq 0.
\end{equation*}
\end{lemma}

In the above $L^q(a,b)$, with $q\in [1,\infty]$, identifies the usual Lebesgue space of real-valued functions defined on the interval $(a,b)$, with $-\infty\le a<b\le +\infty$. A proof of the above lemma can be found in \cite{giu}. 

Next, we will also consider the functional space $C^1(D, \mathbb{R}^d)$ of all continuously differentiable functions from an open subset $D\subset\mathbb{R}^d$ to $\mathbb{R}^d$. For the stability analysis, we will need the following notions from \cite[Chapter 10]{pruss}. 
\begin{definition}[Normally Stable and Normally Hyperbolic Equilibria]
\label{def_norhyp}
Let $D\subset \mathbb{R}^d$ be an open set and $f \in C^1(D, \mathbb{R}^d)$. Consider the nonlinear system of ordinary differential equations (ODEs)
\begin{equation}\label{eq:ode}
\dot z = f(z).
\end{equation} 
Let $\mathcal E\subset D$ denote the set of equilibrium points for \eqref{eq:ode}, i.e., $z\in \mathcal E$ if and only if $f(z)=0$. For $z^*\in \mathcal E$ we assume that the following conditions are satisfied: 
\begin{enumerate}
\item Near $z^*$, $\mathcal{E}$ forms a $C^1$-manifold of dimension $0<m\le d$. This means that there exists an open set $U\subset \mathbb{R}^m$, $0\in U$, and there exists a function $\psi:\; U\to \mathbb{R}^d$ such that $\psi(U)\subset \mathcal E$, $\psi(0)=z^*$, and $\text{Rank }\psi'(0)=m$, where $\psi'(0)\in \mathbb{R}^{m\times m}$ denotes the Jacobian matrix of $\psi$ at $0$. 
\item  $T_{z^*}\mathcal{E} = N(A)$, where $T_{z^*}\mathcal{E}$ denotes the tangent space of $\mathcal{E}$ at $z^*$, $A := f'(z^*)$ is the Fr\'echet derivative of $f$ at $z^*$, and $N(A)$ is the null space of $A$. We may sometimes refer to $A$ as the {\em linearization near the equilibrium} $z^*$. 
\item $\lambda=0$ is a semi-simple eigenvalue of $A$, i.e., $N(A^2)=N(A)$. %$\mathbb{C}^d=N(A)\oplus R(A)$, where $R(A)$ denotes the range of $A$. 
\item $\sigma(A)\cap i\mathbb{R} = \{0\}$. 
\end{enumerate}
Then, the following two definitions holds:
\begin{enumerate}
    \item[i.] $z^*$ is said to be \textbf{normally stable} if $\sigma(A) \setminus \{0\} \subset \mathbb{C}^-$, where $\mathbb{C}^-:=\{\lambda\in \mathbb{C}:\; \mathsf{Re }(\lambda)<0\}$.
    \item[ii.] $z^*$ is said to be \textbf{normally hyperbolic} if $\sigma(A)\cap \mathbb{C}^- \neq \varnothing$ and 
    $\sigma(A)\cap \mathbb{C}^+ \neq \varnothing$, where $\mathbb{C}^+:=\{\lambda\in \mathbb{C}:\; \mathsf{Re }(\lambda)>0\}$.
\end{enumerate}
\end{definition}

\begin{remark}
Geometrically, the condition in the definition of semi-simple eigenvalue is equivalent to requiring that the algebraic and geometric multiplicities of the eigenvalue $\lambda=0$ must coincide.  
\end{remark}
While for linear systems, we could take conditions (i) and (ii) as definition of stable and unstable equilibrium, respectively, for nonlinear systems of ODEs we will consider the following more general definition of stability due to Lyapunov. 

\begin{definition}[Stability in the sense of Lyapunov]
Let $t_0\ge 0$. An equilibrium point $z^*$ of \eqref{eq:ode} is said to be \textbf{stable (in the sense of Lyapunov)} if for every $\varepsilon>0$, there exists $\delta=\delta(\varepsilon)>0$ such that for any $z_0\in G$ satisfying the condition $\|z_0-z^*\|\le \delta$, the (unique) solution to the initial value problem 
\[
\dot z = f(z),\qquad z(t_0)=z_0
\]
satisfies $\|z(t)-z^*\|<\varepsilon$ for every $t\ge t_0$. 

The equilibrium point $z^*$ of \eqref{eq:ode} is said to be \textbf{unstable} if it is not stable (in the sense of Lyapunov). 
\end{definition}

The following {\em linearization principles} hold. 

\begin{theorem}
\label{normalstable}
Consider the nonlinear system of ODEs \eqref{eq:ode}, and assume that $z^*$ is normally stable. Then, the equilibrium $z^*$ is stable (in the sense of Lyapunov) and there exists $\delta > 0$ such that the unique solution to the initial value problem 
\[
\dot z = f(z),\qquad z(0)=z_0
\]
with $z_0\in B_\delta(z^*)$ exists for all $t \geq 0$. Moreover, there exist $z_{\infty} \in \mathcal{E}$ and constants $c_1,\;c_2 > 0$ such that: $$\norm{z(t) - z_{\infty}} \leq c_1e^{-c_2t}$$ for all $t \geq 0$.
\end{theorem}

For a proof, we refer to \cite[Satz 10.4.1]{pruss}

\begin{theorem}
\label{normalhyper}
Consider the nonlinear system of ODEs \eqref{eq:ode}, and assume that $z^*$ is normally hyperbolic. Then the equilibrium $z^*$ is unstable. 

Furthermore, for every sufficiently small $\rho > 0$ there is a $\delta \in (0,\rho]$, such that the solution $z(t)$ with initial value $z_0 \in B_\delta(z^*)$ satisfies exactly one of the following two properties:
\begin{itemize}
    \item $d(z(t^*),\mathcal{E}):=\inf_{z^* \in \mathcal{E}}\norm{z(t^*) - z^*} > \rho$ for some $t^* > 0$; 
    \item $z(t)$ is globally defined for all $t \geq 0$, and there exist $z_{\infty} \in \mathcal{E}$ and constants $c_1,\;c_2 > 0$ such that: $\norm{z(t) - z_{\infty}} \leq c_1e^{-c_2t}$ for all $t \geq 0$.
\end{itemize}
\end{theorem}

A proof of this theorem can be found in \cite[Satz 10.5.1]{pruss}.

\begin{remark}
The above two theorems are generalizations of classical linearization principles for almost linear systems of ODEs, to the case in which the linearization $A$ has zero as eigenvalue (cf. \cite[Section 9.3]{boyce}). 
\end{remark}

\section{Equations of motion}\label{sec:EOM}

%\subsection{Equations of Motion}
We fix an inertial frame $\Sigma_s$ and a body frame $\Sigma_b:=\{G;\mathsf{a}_1,\mathsf{a}_2,\mathsf{a}_3\}$ both with origin at $G$, and $\Sigma_b$ having axes directed along the (orthonormal basis of) eigenvectors of the inertia tensor $J$ of $\mathcal B_1$ with respect to $G$ (these axes are called {\em principal axes of inertia}). We denote $A_1$, $A_2$, $A_3$ the corresponding (positive) eigenvalues of $J$ (called {\em principal moments of inertia} of $\mathcal B_1$), and we assume that $A_1 \leq A_2 \leq A_3$. The inner body $\mathcal B_2$ has inertia tensor $I\mathbf{E_0}$ relative to $G$, where $\mathbf{E_0}$ is the identity tensor. Thus, $J_c := J + I\mathbf{E_0}$ denotes the inertia tensor (with respect to the center of mass $G$) of the whole system of rigid body with spherical damper. 

Let $\omega$ and $\omega_1$ be the angular velocities of $\mathcal B_1$ and $\mathcal B_2$ in the frame $\Sigma_s$, respectively. Similarly, we denote with $\Omega$ and $\Omega_1$ the angular velocities of the outer and inner bodies in the frame $\Sigma_b$, respectively. 

We assume that the interaction between the solids is described by viscous forces (due to the presence of the lubricant) that produce a total torque (relative to $G$) on $\mathcal B_1$ of the form $T:=k(\omega_1-\omega)$ in $\Sigma_s$, with $k$ a positive constant depending on the lubricant density, viscosity and thickness $h$, and on the damper radius $a$ (see \cite[Equation (8.6)]{chern}). %\cite[Equation (7.33)]{evo}).  

We begin with the following balances of angular momentum in the inertial frame $\Sigma_s$
\begin{align}
\label{outerangmom}
    \frac{dm_{total}}{dt} &= 0, \\
\label{innerangmom}
    \frac{dm_{inner}}{dt} &= -k(\omega_1 - \omega).
\end{align}

In the two above equations, $m_{total}$ and $m_{inner}$ are the total angular momentum and the angular momentum of the inner body $\mathcal B_2$ calculated with respect to $G$, respectively.

Given the time dependence of the volumes  $\mathcal B_1$  
(and thus of the inertia tensor of $\mathcal B_1$) in $\Sigma_s$, it is more convenient to rewrite the above equations of motion in the moving frame $\Sigma_b$. In such a frame, $J = \text{diag}(A_1,A_2,A_3)$. We then note that $\Omega = R^T\omega$ and $\Omega_1 = R^T\omega_1$ for some $R = R(t)\in \mathsf{SO}(3)$ for all $t\ge 0$, with $R(0)=\bf{E}_0$. Note that $\Omega$ is the axial vector of $R^T\dot{R}$, i.e., $R^T\dot{R}u = \Omega \times u$ for every $u \in \mathbb{R}^3$. In the body frame $\Sigma_b$, the total angular momentum $M_{total}$ is
\begin{equation*}
    M_{total} :=R^Tm_{total}= J\Omega+I\Omega_1
\end{equation*}

The balance of angular momentum in the body frame $\Sigma_b$ then becomes
\begin{gather}
J\dot{\Omega} + \Omega \times (J \Omega) + I\dot{\Omega}_1 + \Omega \times (I\Omega_1) = 0.  \label{total}
\end{gather}

The inner body is spherical and hence its inertia tensor about the center of mass is just $I\mathbf{E_0}$ (constant) regardless of frame. 
Since $\omega = R\Omega, \omega_1 = R\Omega_1$, the previous equation can be rewritten as follows
\begin{gather}
I\dot{\Omega}_1 + \Omega \times (I\Omega_1) = -k(\Omega_1 - \Omega).\label{inner}
\end{gather}
Subtracting \eqref{inner} from \eqref{total} and collecting the two gives the following system of equations describing the motion of a rigid body with a damper in the moving frame $\Sigma_b$:
\begin{equation*}%\label{EOM}
\begin{split}
    J\dot{\Omega} + \Omega \times (J \Omega) &= k(\Omega_1 - \Omega),\\
    I\dot{\Omega}_1 + \Omega \times (I\Omega_1) &= -k(\Omega_1 - \Omega). 
\end{split}
\end{equation*}
Since $J$ is a positive definite tensor, we invert $J$ to get the following system of differential equations governing the motion of the rigid body $\mathcal B_1$ with the spherical damper $\mathcal B_2$
\begin{equation}
\label{ODE}
    \begin{cases}
      \dot{\Omega} = J^{-1}(k(\Omega_1 - \Omega) - \Omega \times (J \Omega)),\\
      \dot{\Omega}_1 =\displaystyle -\frac{k}{I}(\Omega_1 - \Omega) - \Omega \times \Omega_1.
    \end{cases}       
\end{equation}

Consider the function  $F:\mathbb{R}^3\times\mathbb{R}^3\to\mathbb{R}^6$ such that 
\begin{equation}\label{eq:F}
\qquad F(\Omega, \Omega_1):=\left[\begin{matrix}
J^{-1}(k(\Omega_1 - \Omega) - \Omega \times (J \Omega))
\\
\displaystyle -\frac{k}{I}(\Omega_1 - \Omega) - \Omega \times \Omega_1
\end{matrix}\right]. %\in \mathbb{R}^6. 
\end{equation} 
Since $F$ is continuously differentiable in $\mathbb{R}^3\times\mathbb{R}^3$, for any initial data $(\Omega_0,\Omega_{10}) \in \mathbb{R}^3\times\mathbb{R}^3$, the initial value problem 
\[
\frac{d}{dt}\left[\begin{matrix}
\Omega
\\
\Omega_1\end{matrix}\right]=F(\Omega, \Omega_1),\qquad 
\left[\begin{matrix}
\Omega(0)
\\
\Omega_1(0)\end{matrix}\right]=\left[\begin{matrix}
\Omega_0
\\
\Omega_{10}\end{matrix}\right]
\]
has a unique solution defined on the maximal existence interval $[0,t_+)$ for some $t_+=t_+(\Omega_0,\Omega_{10})>0$. We will soon see that $t_+=+\infty$. While we cannot explicitly solve \eqref{ODE} in all cases, we can derive some useful properties of the solutions to \eqref{ODE} like the balances of the kinetic energy and total angular momentum. 
 
First consider the kinetic energy, defined by 
\[
V(\Omega,\Omega_1) := \frac{1}{2}(\inner{\Omega}{J\Omega} + I\inner{\Omega_1}{\Omega_1}). 
\]
By taking the time derivative of $V$ and using equations \eqref{ODE}, we get
\begin{align*}
    \frac{d V}{dt} &= \frac{1}{2}(2\inner{\Omega}{J\dot\Omega} + 2I\inner{\Omega_1}{ \dot{\Omega}_1}) \\
    &= \inner{\Omega}{k(\Omega_1 - \Omega) - \Omega \times (J \Omega)} + \inner{\Omega_1}{-k(\Omega_1 - \Omega) - I(\Omega \times \Omega_1)} \\
    &= -k\norm{\Omega_1 - \Omega}^2.
\end{align*}

Hence the rate of change of the kinetic energy along the solutions of  \eqref{ODE} reads as follows
\begin{equation}
\label{tderiv-kin}
    \frac{d}{dt}V(\Omega(t),\Omega_1(t)) = -k\norm{\Omega_1(t) - \Omega(t)}^2. 
\end{equation}
The above equation physically expresses the balance of the kinetic energy of the whole system.

Another important balance is that of the angular momentum. Consider the quantity $K^2 = K^2(\Omega,\Omega_1):=\norm{J\Omega + I\Omega_1}^2$. %, which is the squared norm of the total angular momentum (with respect to $G$) in the body frame. 
By taking the time derivative of $K^2$ we get the following equation:
\begin{align}
    \frac{dK^2}{dt} &= 2\inner{J\Omega + I\Omega_1}{J\dot{\Omega} + I\dot{\Omega}_1} \nonumber\\
    &= 2\inner{J\Omega + I\Omega_1}{k(\Omega_1 - \Omega) - \Omega \times (J \Omega) -k(\Omega_1 - \Omega) - I(\Omega \times \Omega_1)}\nonumber\\
    &= -2\inner{J\Omega + I\Omega_1}{\Omega \times (J \Omega + I\Omega_1)}\nonumber\\
    &= 0. \label{tderiv-ang}
\end{align}

Thus the quantity $K^2$ is constant along solutions of \eqref{ODE}. We summarize the above results in the next proposition. 
\begin{proposition}\label{prop:balances}
Let $(\Omega,\Omega_1)$ be a solution to the system of ODEs \eqref{ODE}. Then,  
\begin{itemize}
\item[i.] $(\Omega,\Omega_1)$ satisfies \eqref{tderiv-kin}. 
\item[ii.] $(\Omega,\Omega_1)$ satisfies \eqref{tderiv-ang}. 
\end{itemize}
\end{proposition}

Since $J$ is positive definite and $I$ is a positive constant, a corollary of the statement (i) is that the functions $\Omega(t), \Omega_1(t)$ are uniformly bounded for all $t \geq 0$, and by the classical continuation theorem for ODEs (see \cite[Theorem 4.1]{coddington}), it follows that $t_+=+\infty$. In addition, Proposition \ref{prop:balances}(ii) allows us to define invariant sets for all solutions of \eqref{ODE}. Indeed, for a given initial condition $(\Omega_0,\Omega_{10})$, let $M = K^2(\Omega_0,\Omega_{10})$ and define the set $A(M)$: $$A(M) = \{(\omega,\omega_1) \in \mathbb{R}^6 : K^2(\omega,\omega_1) = M\}$$ Then the solution to \eqref{ODE} corresponding to $(\Omega_0,\Omega_{10})$ must stay within $A(M)$ for all times. Furthermore, since $K^2$ is a continuous function and $A(M)$ is the preimage of a compact set in $\mathbb{R}^6$, we have the following property. 
\begin{lemma}
\label{A-compact}
The set $A(M)$ is compact.
\end{lemma}

\section{Equilibria and their stability properties}\label{sec:equilibria}
\subsection{Equilibria}
We begin this section by characterizing the equilibria of \eqref{ODE}  in the following theorem, first proved in \cite{evo}.
\begin{theorem}
\label{equilibria}
The point $(\Omega^*,\Omega_1^*) \in \mathbb{R}^3\times\mathbb{R}^3$ is an equilibrium point of \eqref{ODE} if and only if $\Omega^* = \Omega_1^*$ and $\Omega^*$ is either the zero vector in $\mathbb{R}^3$ or an eigenvector of $J$.
\end{theorem}
\begin{proof}
Assume that $F(\Omega^*,\Omega_1^*) = 0$, then from \eqref{ODE} we have:
\begin{equation*}\begin{cases}
    \Omega^* \times (J \Omega^*) = k(\Omega_1^* - \Omega^*)\\
    \Omega^* \times (I\Omega_1^*) = -k(\Omega_1^* - \Omega^*)
\end{cases}\end{equation*}

Taking the dot product of both sides of the second equation by $(\Omega_1^* - \Omega^*)$, it gives $-k\norm{\Omega_1^* - \Omega^*}^2 = 0$, and thus we must have $\Omega^* = \Omega_1^*$. Replacing the latter in the first equation of the above system, it follows that $\Omega^* \times (J \Omega^*) = 0$ which holds if and only if $\Omega^*$ is either the zero vector in $\mathbb{R}^3$ or an eigenvector of $J$.

The converse implication is trivially true. 
\end{proof}

\begin{remark}
\label{equilibria-remark}
Define $\{\eu_i\}_{i=1}^6$ to be the standard basis vectors of $\mathbb{R}^6$. From our assumption that $A_1 \leq A_2 \leq A_3$, there are four possible cases for the set of equilibria $\mathcal{E}$:
\begin{enumerate}
    \item $A_1 < A_2 < A_3 \implies \mathcal{E} = \text{span}\{\eu_1 + \eu_4\}\cup\text{span}\{\eu_2 + \eu_5\}\cup\text{span}\{\eu_3 + \eu_6\}$
    \item $A_1 = A_2 < A_3 \implies \mathcal{E} = \text{span}\{\eu_1 + \eu_4, \eu_2 + \eu_5\}\cup\text{span}\{\eu_3 + \eu_6\}$
    \item $A_1 < A_2 = A_3 \implies \mathcal{E} = \text{span}\{\eu_1 + \eu_4\}\cup\text{span}\{\eu_2 + \eu_5, \eu_3 + \eu_6\}$
    \item $A_1 = A_2 = A_3 \implies \mathcal{E} = \text{span}\{\eu_1 + \eu_4, \eu_2 + \eu_5, \eu_3 + \eu_6\}$
\end{enumerate}
\end{remark}

\subsection{Nonlinear stability}
For the study of the stability, we will use the linearization principles in Theorem \ref{normalstable} and Theorem \ref{normalhyper}. Let us start by linearizing the equations \eqref{ODE} around a (nonzero\footnote{The aim of this paper is to characterize the long-time behaviour of generic trajectories, corresponding to nonzero initial conditions, then any attained equilibrium must be nonzero by the conservation of the total angular momentum (see Proposition \ref{prop:balances}).}) equilibrium $(\Omega^*,\Omega_1^*\equiv \Omega^*)$. By Theorem \ref{equilibria} there exists $\lambda^*\in\{A_1,A_2,A_3\}$ such that $(J-\lambda^*\mathbf{E_0})\Omega^*=0$. We will also denote $\mathsf{eigen}_J(\lambda^*)$ the eigenspace corresponding to the eigenvalue $\lambda^*$ of $J$. 

For every $(\Omega,\Omega_1)\in \mathbb{R}^6$, the Gateaux derivative of $F$ at $(\Omega^*,\Omega_1^*\equiv \Omega^*)$ is given by 
\begin{equation}\label{eq:linearization}\begin{split}
\frac{d}{dt}&F(\Omega^*+t\Omega,\Omega_1^*+t\Omega_1)\vert_{t=0}
\\
&=\left[\begin{matrix}
J^{-1}(k(\Omega_1 - \Omega) - \Omega^*\times (J-\lambda^*\mathbf{E_0})\Omega)
\\
-k/I(\Omega_1 - \Omega) -\Omega^* \times (\Omega_1-\Omega)
\end{matrix}\right] = L^*\left[\begin{matrix}
\Omega
\\
\Omega_1
\end{matrix}\right],
\end{split}\end{equation}
for some linear transformation $L^*=L^*(J,k,I,\Omega^*,\Omega^*_1):\; \mathbb{R}^3\times\mathbb{R}^3\to \mathbb{R}^3\times\mathbb{R}^3$. We will refer to $L^*$ as the {\em linearization at the equilibrium} $(\Omega^*,\Omega_1^*\equiv \Omega^*)$. 

The rest of this section concerns the above linear transformation. 
We start with the following lemma asserting that the null space of $L^*$ is characterized by equilibria. 
\begin{lemma}
\label{kernel}
$N(L^*)=\{v=(\Omega,\Omega_1)\in \mathcal E:\; \Omega=\Omega_1\in \mathsf{eigen}_J(\lambda^*)\}$. 
\end{lemma}
\begin{proof}
Let $v=(\Omega,\Omega_1) \in \mathbb{R}^3\times\mathbb{R}^3$ be such that $L^*v = 0$. From \eqref{eq:linearization}, the condition $L^*v = 0$ is equivalent to the following system of equations:
\begin{equation}\label{eq:L=0}\begin{split}
&k(\Omega_1 - \Omega) -\Omega^*\times (J-\lambda^*\mathbf{E_0})\Omega=0,
\\
&-\frac kI (\Omega_1 - \Omega) - \Omega^* \times (\Omega_1-\Omega)=0.
\end{split}\end{equation}
Dot multiplying both sides of the second equation by $\Omega -\Omega_1$, we find that $\norm{\Omega-\Omega_1}=0$, thus implying that $\Omega=\Omega_1$. 
As a consequence, the first equation of the above system nows reads 
\[
\Omega^*\times[(J-\lambda^*\mathbf{E_0})\Omega]=0,
\]
and we can infer the existence of a scalar $\bar \lambda\in \mathbb{R}$ such that
\[
(J-\lambda^*\mathbf{E_0})\Omega=\bar\lambda \Omega^*.
\]
Let us take the dot product of the latter displayed equation by $\Omega^*$ and we find that 
\[
\bar\lambda \norm{\Omega^*}^2=\inner{\Omega}{J\Omega^*}-\lambda^*\inner{\Omega}{\Omega^*}=0.
\]
Thus, $\bar\lambda=0$ and we can conclude that also $(J-\lambda^*\mathbf{E_0})\Omega=0$. 

Summarizing, we have shown that $v=(\Omega,\Omega_1) \in N(L^*)$ satisfies $\Omega=\Omega_1$ and $(J-\lambda^*\mathbf{E_0})\Omega=0$. Hence, by Theorem \ref{equilibria}, we conclude that $v=(\Omega,\Omega_1)\in \mathcal E$ with $\Omega=\Omega_1\in \mathsf{eigen}_J(\lambda^*)$. 

The converse inclusion is immediately verified. 
\end{proof}

Due to the presence of the zero eigenvalue, we cannot apply the classical linearization principles for almost linear systems (see \cite[Section 9.3]{boyce}). This motivates our attempt to use Theorem \ref{normalstable} and Theorem  \ref{normalhyper} instead. The remainder of this section focuses on verifying that our equilibria satisfy the conditions of those theorems, starting with the following lemmas.
\begin{lemma}
\label{noimag}
$\sigma(L^*)\cap i\mathbb{R} = \{0\}$. 
\end{lemma}
\begin{proof}
This proof requires verifying that in all cases we cannot have a purely imaginary eigenvalue. Given \eqref{eq:linearization}, an equivalent proof of the above statement is to show that the following linear system of ODEs 
\begin{equation}\label{eq:linearODEs}
\begin{split}
J\dot \Omega &= k(\Omega_1 - \Omega) - \Omega^*\times (J-\lambda^*\mathbf{E_0})\Omega
\\
I\dot \Omega_1 &= -k(\Omega_1 - \Omega) -I\Omega^* \times (\Omega_1-\Omega)
\end{split}\end{equation} 
does not admit any periodic solution other than the trivial one. Let us argue by contradiction, and assume that the above system of ODEs admits a nontrivial periodic solution $(\Omega, \Omega_1)$ with period $T>0$ (i.e., $(\Omega(t), \Omega_1(t))=(\Omega(t+T), \Omega_1(t+T))$ for every $t\in \mathbb R$). Without loss of generality, we will assume that 
\begin{equation}\label{eq:zeroaverage}
\int^T_0\Omega(t)\; dt=0,\qquad \int^T_0\Omega_1(t)\; dt=0.
\end{equation}
Let us denote $K:=J\Omega+I\Omega_1$. Adding the two equations in \eqref{eq:linearODEs} side by side, we find the following equation for $K$
\begin{equation}\label{eq:linearangular}
\dot{K}=-\Omega^*\times K+\Lambda^*\Omega^*\times\Omega 
\end{equation} 
where $\Lambda^*:=\lambda^*+I$. Let us take the dot product of \eqref{eq:linearODEs}$_{1}$ by $\Omega$, we obtain 
\begin{equation}\label{eq:linearenergy1}
\frac{d}{dt}\left(\frac 12 \inner{\Omega}{J\Omega}\right)=k\inner{\Omega_1-\Omega}{\Omega}+\inner{\Omega^*\times \Omega}{J\Omega}.
\end{equation}
Similarly, let us take the dot product of \eqref{eq:linearODEs}$_{2}$ by $\Omega_1$, we get 
\begin{equation}\label{eq:linearenergy2}
\frac{d}{dt}\left(\frac I2 \norm{\Omega_1}^2\right)=-k\inner{\Omega_1-\Omega}{\Omega_1}+\inner{\Omega^*\times \Omega}{I\Omega_1}.
\end{equation}
Finally, let us dot multiply \eqref{eq:linearangular} by $K$ to find 
\begin{equation}\label{eq:linearangular2}
\frac{d}{dt}\left(\frac 12 \norm{K}^2\right)=\Lambda^*\inner{\Omega^*\times \Omega}{K}.
\end{equation}
Take \eqref{eq:linearangular2} and subtract the sum of \eqref{eq:linearenergy1} and \eqref{eq:linearenergy2} multiplied by $\Lambda^*$, we obtain 
\[
\frac12\frac{d}{dt}\left[\norm{K}^2-\Lambda^*(\inner{\Omega}{ J\Omega}+I\norm{\Omega_1}^2)\right]=-k\Lambda^*\norm{\Omega-\Omega_1}^2.
\]
Integrating the latter displayed equation over a period $(t,t+T)$, we find that 
\[
\int^{t+T}_t\norm{\Omega(s)-\Omega_1(s)}^2\; ds=0
\]
for all $t\in \mathbb{R}$. Hence, $\Omega(t)=\Omega_1(t)$ for every $t$. Replacing the latter information in \eqref{eq:linearODEs}$_2$, we find that $\dot\Omega_1=0$, and thus $\Omega(t)=\Omega_1(t)\equiv 0$ by \eqref{eq:zeroaverage}. 
\end{proof}

\begin{lemma}
\label{semisimple}
The zero eigenvalue of $L^*$ is semi-simple.
\end{lemma}
\begin{proof}
We only need to show that $N((L^*)^2)\subset N(L^*)$ as the converse inclusion is trivially satisfied. Let $w=(\Omega,\Omega_1)\in N((L^*)^2)$, and set $v=L^*w$. We want to show that $v=0$. Since $v\in N(L^*)$, by Lemma \ref{kernel}, $v=(\tilde\Omega,\tilde\Omega)$ for some $\tilde\Omega\in \mathsf{eigen}_J(\lambda^*)$. By \eqref{eq:linearization}, we have that 
\begin{equation}\label{eq:nulleq}\begin{split}
k(\Omega_1 - \Omega) - \Omega^*\times (J-\lambda^*\mathbf{E_0}\Omega)&=\lambda^*\tilde\Omega,
\\
-k(\Omega_1 - \Omega) -\Omega^* \times I(\Omega_1-\Omega)&=I\tilde\Omega.
\end{split}\end{equation}
Adding side by side the latter displayed equations, we find 
\begin{equation}\label{eq:nulleq2}
-\Omega^*\times K+\Lambda^*\Omega^*\times\Omega=\Lambda^*\tilde\Omega,
\end{equation}
where we recall that  $K=J\Omega+I\Omega_1$ and $\Lambda^*:=\lambda^*+I$. From \eqref{eq:nulleq2}, it immediately follows that 
\begin{equation}\label{eq:orthogonal1}
\inner{\tilde\Omega}{\Omega^*}=0.
\end{equation} 
Using the latter in \eqref{eq:nulleq}$_2$ dot multiplied by $\Omega^*$, we also get 
\begin{equation}\label{eq:orthogonal2}
\inner{\Omega_1-\Omega}{\Omega^*}=0.
\end{equation}
Let us take the cross product from the left of \eqref{eq:nulleq2} by $\Omega^*$, we find 
\[
\Lambda^*\Omega^*\times\tilde \Omega=-\Omega^*\times[\Omega^*\times(K-\Lambda^*\Omega)]=-[\Omega^*(K-\Lambda^*\Omega)]\Omega^*+\norm{\Omega^*}^2(K-\Lambda^*\Omega).
\]
Now, we dot multiply the latter displayed equality by $\tilde\Omega$ and use \eqref{eq:orthogonal1} together with the fact that $\lambda^*\tilde\Omega=J\tilde \Omega$ with $J$ a symmetric tensor, to get
\[\begin{split}
0&=\norm{\Omega^*}^2\inner{K-\Lambda^*\Omega}{\tilde\Omega}
=\norm{\Omega^*}^2\inner{J\Omega+I\Omega_1-\lambda^*\Omega-I\Omega}{\tilde\Omega}
\\
&=\norm{\Omega^*}^2(\inner{\Omega}{J\tilde\Omega}+I\inner{\Omega_1-\Omega}{\tilde\Omega}-\inner{\Omega}{J\tilde\Omega})
=\norm{\Omega^*}^2I\inner{\tilde\Omega}{\Omega_1-\Omega}.
\end{split}\]
From which, it follows that 
\begin{equation}\label{eq:orthogonal3}
\inner{\tilde\Omega}{\Omega_1-\Omega}=0.
\end{equation}
Let us take the dot product of \eqref{eq:nulleq}$_2$ by $\tilde \Omega$, using the  three orthogonality conditions \eqref{eq:orthogonal1}, \eqref{eq:orthogonal2} and \eqref{eq:orthogonal3}, we obtain that 
\begin{equation}\label{eq:orthogonal4}
\norm{\tilde\Omega}=\norm{\Omega^*}\norm{\Omega_1-\Omega}.
\end{equation}
It remains to show that $\norm{\Omega_1-\Omega}=0$. Let us take the dot product of \eqref{eq:nulleq}$_{1}$ by $\Omega$ and of \eqref{eq:nulleq}$_2$ by $\Omega_1$, respectively, we find the following two equations:
\[\begin{split}
&k\inner{\Omega_1-\Omega}{\Omega}+\inner{\Omega^*\times\Omega}{J\Omega}=\lambda^*\inner{\tilde\Omega}{\Omega},
\\
&-k\inner{\Omega_1-\Omega}{\Omega_1}+\inner{\Omega^*\times\Omega}{I\Omega_1}=I\inner{\tilde\Omega}{\Omega_1}.
\end{split}\]
Summing side by side the above equations, we obtain 
\begin{equation}\label{eq:nulleq3}
-k\norm{\Omega_1-\Omega}^2+\inner{\Omega^*\times\Omega}{K}=\inner{\tilde\Omega}{K}.
\end{equation}
On the other side, dot multiplying \eqref{eq:nulleq2} by $K$, we discover that 
\begin{equation}\label{eq:nulleq4}
\Lambda^*\inner{\Omega^*\times\Omega}{K}=\Lambda^*\inner{\tilde \Omega}{K}.
\end{equation}
From \eqref{eq:nulleq3} and \eqref{eq:nulleq4}, it immediately follows that $\norm{\Omega_1-\Omega}=0$, and this concludes our proof. 
\end{proof}

We are now ready to state and prove the main result about the spectral stability properties of the equilibria. 

\begin{theorem}
\label{pruss-theorem}
Let $(\Omega^*,\Omega_1^*)$ be an equilibrium point of \eqref{ODE} such that $\Omega^* = \Omega_1^*$ is an eigenvector of $J$ corresponding to $\lambda^*\in \{A_1,A_2,A_3\}$\footnote{Recall that we have denoted with  $A_1, A_2, A_3$ the eigenvalues of $J$, and we have assumed that $A_1 \leq A_2 \leq A_3$.}. Then the equilibrium  $(\Omega^*,\Omega_1^*)$ is normally stable if $\lambda^*=A_3$, otherwise it is normally hyperbolic.
\end{theorem}
\begin{proof}
From Lemma \ref{semisimple} we know that 0 is a semi-simple eigenvalue. We now verify that for an equilibrium $(\Omega^*,\Omega_1^*)$, the set of equilibria in a neighbourhood of that point forms a $C^1$ manifold, and that the tangent space of said manifold at $(\Omega^*,\Omega_1^*)$ is equal to the null space of the linearization at that point.

In Theorem \ref{equilibria} and Remark \ref{equilibria-remark}, we saw that a given equilibrium $(\Omega^*,\Omega_1^*) =: z^*$ lies along a subspace of either 1, 2, or 3 dimensions. Without loss of generality, assume $z^*$ is along the 1-dimensional subspace $\text{span}\{\eu_3 + \eu_6\}$. We define an open set $U = \mathbb{R}$ and a $C^1$ function $\Psi: U \rightarrow \mathbb{R}^6$ by: $$\Psi(x) = x(\eu_3 + \eu_6) + z^*,\quad x\in \mathbb{R}.$$

Then we can directly observe that $\Psi(U) \subset \mathcal{E}, \Psi(0) = z^*$. We can also compute the derivative of $\Psi$ at $0$ to be $\Psi'(0) = \eu_3 + \eu_6$, and hence $\text{rank}(\Psi'(0)) = 1$. This verifies that near $z^*$, the set of equilibria form a 1-dimensional $C^1$ manifold. $\Psi$ defines a parameterization of points along this manifold, and the tangent space at $z^*$ is simply the span of $\Psi'(0)$: $$T_{z^*}\mathcal{E} = \text{span}\{\eu_3 + \eu_6\}.$$ 

We recognize from Lemma \ref{kernel} that the above subspace is the null space of our linearization when taken about $z^* \in \text{span}\{\eu_3 + \eu_6\}$. %, thus verifying it in the case $z^*$ is along a 1-dimensional subspace. 
Similar constructions can be used to verify the same steps in the higher dimensional cases. The remainder of the proof is then devoted to characterize the location of the spectrum of $L^*$. 

Let us assume that $\lambda^*=A_3$, and let $\lambda\in \sigma(L^*)$, $\lambda\ne 0$. We will show that $\mathsf{Re }(\lambda)<0$. Let $(\Omega,\Omega_1)\in \mathbb{C}^6$ be the eigenvector of $L^*$ corresponding to $\lambda$. By \eqref{eq:linearization}, $\lambda$ and $(\Omega,\Omega_1)$ satisfy the following algebraic system of equations 
\begin{equation}\label{eq:eigenvalue_p}
\left\{\begin{aligned}
&k(\Omega_1 - \Omega) - \Omega^*\times (J-\lambda^*\mathbf{E_0})\Omega=\lambda J\Omega,
\\
& -k(\Omega_1 - \Omega) -\Omega^* \times I(\Omega_1-\Omega)=\lambda I\Omega_1.
\end{aligned}\right.
\end{equation}
Taking the inner product in $\mathbb{C}^3$ of \eqref{eq:eigenvalue_p}$_2$ by $\Omega_1$, we obtain
\begin{equation}\label{eq:eigen4}
-k\inner{\Omega_1-\Omega}{\overline{\Omega_1}}+2iI\inner{\mathsf{Re}(\Omega_1)\times\mathsf{Im}(\Omega_1)}{\Omega^*}
+I\inner{\Omega\times\overline{\Omega_1}}{\Omega^*}=\lambda I\norm{\Omega_1}_{\mathbb{C}^3}^2.
\end{equation}
Now, take the inner product in $\mathbb{C}^3$ of \eqref{eq:eigenvalue_p}$_2$ by $\Omega$, we get 
\begin{equation}\label{eq:eigen5}
-k\inner{\Omega_1-\Omega}{\overline{\Omega}}-2iI\inner{\mathsf{Re}(\Omega)\times\mathsf{Im}(\Omega)}{\Omega^*}
+I\inner{\overline{\Omega}\times\Omega_1}{\Omega^*}=\lambda I \inner{\Omega_1}{\overline\Omega}. 
\end{equation}
Subtract \eqref{eq:eigen5} from \eqref{eq:eigen4}, and take the real part of the resulting equation, we find that 
\begin{equation}\label{eq:eigen6}
\mathsf{Re}(\lambda)\norm{\Omega_1}_{\mathbb{C}^3}^2-\mathsf{Re}[\lambda\inner{\Omega_1}{\overline{\Omega}}]=-\frac kI \norm{\Omega-\Omega_1}_{\mathbb{C}^3}^2.
\end{equation}
Adding side by side the equations in \eqref{eq:eigenvalue_p}, we find that the vector  $K=J\Omega+I\Omega_1$ satisfies the following equation 
\begin{equation}\label{eq:eigen7}
\lambda K=-\Omega^*\times K+\Lambda^*\Omega^*\times \Omega,
\end{equation}
where $\Lambda^*=A_3+I$ in this case. Let us consider the inner product in $\mathbb{C}^3$ of \eqref{eq:eigen7} by $\Omega$, we obtain 
\begin{equation}\label{eq:eigen8}
\lambda \inner{K}{\overline\Omega}=\inner{\overline{\Omega}\times K}{\Omega^*}+\Lambda^*\inner{\Omega\times\overline{\Omega}}{\Omega^*}.
\end{equation}
Let us take the inner product in $\mathbb{C}^3$ of \eqref{eq:eigen7} by $K$, we get 
\begin{equation}\label{eq:eigen9}
\lambda\norm{K}_{\mathbb{C}^3}=-\inner{K\times\overline K}{\Omega^*}+\Lambda^*\inner{\Omega\times\overline{K}}{\Omega^*}.
\end{equation}
We now multiply both sides of \eqref{eq:eigen8} by $\Lambda^*$, subtract this new equation from \eqref{eq:eigen9} and take the real part of what we obtained to find the following important equality
\begin{equation*}%\label{eq:eigen10}
\mathsf{Re}(\lambda)\norm{K}_{\mathbb{C}^3}=\Lambda^*\mathsf{Re}(\lambda)\inner{J\Omega}{\overline\Omega}+\Lambda^*I\mathsf{Re}[\lambda \inner{\Omega_1}{\overline\Omega}].
\end{equation*}
Using \eqref{eq:eigen6} in the last term of the latter displayed equality, we get 
\[
\mathsf{Re}(\lambda)[\Lambda^*\inner{J\Omega}{\overline\Omega}+\Lambda^*I\norm{\Omega_1}_{\mathbb{C}^3}^2-\norm{K}_{\mathbb{C}^3}]=-k\Lambda^*\norm{\Omega-\Omega_1}.
\]
We notice that the term in the squared parentheses is a real number , it remains to show that it is positive which is what we will prove next. Writing $\Omega=(p,q,r)^T\in \mathbb{C}^3$ and $\Omega_1=(p_1,q_1,r_1)^T\in \mathbb{C}^3$, we can then write more explicitly
\[\begin{split}
\Lambda^*\inner{J\Omega}{\overline\Omega}+\Lambda^*I\norm{\Omega_1}_{\mathbb{C}^3}^2&-\norm{K}_{\mathbb{C}^3}
\\
=&I\left(A_1\norm{p-p_1}_{\mathbb C}^2+A_2\norm{q-q_1}_{\mathbb C}^2+A_3\norm{q-q_1}_{\mathbb C}^2\right)
\\
&(A_3-A_1)(A_1+I)\norm{p}_{\mathbb C}^2+(A_3-A_2)(A_2+I)\norm{q}_{\mathbb C}^2.
\end{split}
\]
It is then clear that the right-hand side of the latter equality is positive since $I>0$ and $A_3\ge A_2\ge A_1>0$ by assumption. 

Let us now show that if $\lambda^*\in\{A_1,A_2\}$, then $\sigma(L^*)\cap \mathbb{C}^+ \neq \varnothing$. Note that the conditions (on the spectrum) characterizing normally stable and normally hyperbolic equilibria are not mutually exclusive. In particular, normally hyperbolic equilibria require both a stable and an unstable part of the spectrum of the linearization (see Definition \ref{def_norhyp}). We will in fact prove more than what is required, we will demonstrate the following two important facts: 
\begin{enumerate}
\item[(F1)] If $A_1\le \lambda^*=A_2<A_3$, then $L^*$ has only one positive eigenvalue.
\item[(F2)] If $\lambda^*=A_1<A_2\le A_3$, then $L^*$ has only two eigenvalues having positive real part.
\end{enumerate}
Before proving properties (F1) and (F2), let us make some observations. Let us consider the diagonal tensor $J_\mu$ (in the basis $\mathsf{a}_1$, $\mathsf{a}_2$ and $\mathsf{a}_3$ of the eigenvectors of $J$)
\begin{equation}\label{eq:J_mu}
J_\mu:=\left[\begin{matrix}
\mu_1 & 0 & 0
\\
0 & \mu & 0
\\
0 & 0 & \mu_3\end{matrix}\right],
\end{equation}
with $0<\mu_1<\mu,\mu_3\in \mathbb{R}$, and assume that $\lambda^*=\mu_3$. Let us denote with $L(\mu)$ the linear transformation defined in equation \eqref{eq:linearization} with $J$ replaced by $J_\mu$. From the above calculations, Lemma \ref{kernel} and Lemma \ref{semisimple}, we have the following two cases:
\begin{enumerate}
\item[(i)] If $\mu_1<\mu<\mu_3$, then $\lambda=0$ is an eigenvalue of $L(\mu)$ with multiplicity one and corresponding eigenvector $(\Omega^*,\Omega^*)$, where $\Omega^*\in \text{span}\{\mathsf{a}_3\}$. All the other eigenvalues of $L(\mu)$ have negative real part. 
\item[(ii)] If $\mu_1<\mu=\mu_3$, then $\lambda=0$ is an eigenvalue of $L(\mu)$ with multiplicity two and corresponding eigenvector $(\Omega^*,\Omega^*)$, where $\Omega^*\in \text{span}\{\mathsf{a}_2,\mathsf{a}_3\}$. All the other eigenvalues of $L(\mu)$ have negative real part. 
\end{enumerate}
The above remarks suggest that a nonzero eigenvalue $\lambda$ of $L(\mu)$ goes from having $\mathsf{Re}(\lambda)<0$ to be identically zero as $\mu\nearrow \mu_3$. We will then prove that $\sigma(L(\mu))$ crosses the imaginary axis with positive speed as as $\mu\nearrow \mu_3$, and this would be enough to conclude the existence of eigenvalues with positive real part when $\mu_3<\mu$. 

Let us start with the proof of property (F1). Assume that $\lambda^*=\mu_3$, and let $\Omega^*=\Omega_1^*=\alpha^*\,\mathsf{a}_3$ for some $\alpha^*\in \mathbb{R}\setminus\{0\}$. Let us consider the eigenvalue problem, in the new variables $(\lambda(\mu),\Omega(\mu),\Omega_1(\mu))$, obtained from \eqref{eq:eigenvalue_p} by replacing $J$ with $J_\mu$, defined in \eqref{eq:J_mu}: 
\begin{equation}\label{eq:eigenvalue_mu}
\left\{\begin{aligned}
&k(\Omega_1(\mu) - \Omega(\mu)) - \alpha^*\,\mathsf{a}_3\times (J_\mu-\lambda^*\mathbf{E_0})\Omega(\mu)&&=\lambda(\mu) J_\mu\Omega(\mu),
\\
& -k(\Omega_1(\mu) - \Omega(\mu)) -\alpha^*\,\mathsf{a}_3 \times I(\Omega_1(\mu)-\Omega(\mu))&&=\lambda(\mu) I\Omega_1(\mu).
\end{aligned}\right.
\end{equation} 
We can rewrite the above linear algebraic system in the compact form 
\begin{equation}\label{eq:eigenvalue_mu_c}
(L(\mu)-\lambda(\mu))(\Omega(\mu),\Omega_1(\mu))
=0_{\,\mathbb{R}^6}. 
\end{equation}
Consider the map
\begin{equation}\label{eq:implicit_G}
\begin{split}
(\mu,(\lambda, \Omega,\Omega_1))&\in (\mu_3-\delta,\mu_3+\delta)\times(\mathbb{R}\times\mathbb{R}^3\times\mathbb{R}^3)
\\
\mapsto G(\mu,(\lambda,\Omega,\Omega_1))& :=
((L(\mu)-\lambda)(\Omega,\Omega_1),\inner{\Omega}{\mathsf{a}_3},\norm{\Omega}^2-1)\in \mathbb{R}^3\times\mathbb{R}\times\mathbb{R}. 
\end{split}
\end{equation}
Note that $G$ has continuous partial derivatives and $G(\mu_3,(0,\mathsf{a}_2,\mathsf{a}_2))=0_{\,\mathbb{R}^5}$. We will show that, if $\delta$ is sufficiently small, for each $\mu\in (\mu_3-\delta,\mu_3+\delta)$ there exist (unique) $(\lambda(\mu),\Omega(\mu),\Omega_1(\mu))$ such that $G(\mu,(\lambda(\mu),\Omega(\mu),\Omega_1(\mu)))=0_{\,\mathbb{R}^5}$. We will use the Implicit Function Theorem. To this end, we need to show that the Fr\'echet derivative  of $G$ with respect to $h\equiv (\lambda, \Omega,\Omega_1)$ at $(\mu_3,(0,\mathsf{a}_2,\mathsf{a}_2))$, given by 
\[
D_hG(\mu_3,(0,\mathsf{a}_2,\mathsf{a}_2))[\hat{\lambda},\hat{\Omega},\hat{\Omega}_1]=(L(\mu_3)(\hat{\Omega},\hat{\Omega}_1)-\hat{\lambda}(\mathsf{a}_2,\mathsf{a}_2),\inner{\hat{\Omega}}{\mathsf{a}_3},2\inner{\hat{\Omega}}{\mathsf{a}_2}),
\]
is invertible. As $D_hG(\mu_3,(0,\mathsf{a}_2,\mathsf{a}_2))$ is itself a linear transformation, to show that it is invertible, it is enough that it is injective. Assume that $D_hG(\mu_3,(0,\mathsf{a}_2,\mathsf{a}_2))[\hat{\lambda},\hat{\Omega},\hat{\Omega}_1]=0_{\,\mathbb{R}^5}$, that is 
\begin{equation}\label{eq:G=0}
\begin{split}
L(\mu_3)\left[\begin{matrix}
\hat \Omega
\\
\hat{\Omega}_1\end{matrix}\right]
&=\hat\lambda\left[\begin{matrix}
\mathsf{a}_2
\\
\mathsf{a}_2\end{matrix}\right],
\\
\inner{\hat{\Omega}}{\mathsf{a}_3}&=0
\\
\inner{\hat{\Omega}}{\mathsf{a}_2}&=0
\end{split}
\end{equation}
we will show that $(\hat{\lambda},\hat{\Omega},\hat{\Omega}_1)=(0,0_{\mathbb{R}^3},0_{\mathbb{R}^3})$. Note that, since $\mu=\mu_3$, by Lemma \ref{kernel}
\[
\hat\lambda(\mathsf{a}_2,\mathsf{a}_2)\in N(L(\mu_3))=\{v=(\Omega,\Omega_1)\in \mathcal E:\; \Omega=\Omega_1\in \mathsf{eigen}_{J_{\mu_3}}(\lambda^*)\equiv\text{span}\{\mathsf{a}_2,\mathsf{a}_3\}\}.
\] 
Since $0$ is a semi-simple eigenvalue (Lemma \ref{semisimple}), we then have that 
$(\hat\Omega,\hat{\Omega}_1)\in N(L^2(\mu_3))=N(L(\mu_3))$, implying that $
\hat\Omega=\hat{\Omega}_1=\hat q\mathsf{a}_2+\hat r\mathsf{a}_3$ 
for some $\hat q,\hat r\in \mathbb{R}$. From the last two equations in \eqref{eq:G=0}, we can the conclude that $\hat q=\hat r=0$, which in turn implies that also $\hat \lambda=0$ by \eqref{eq:G=0}$_1$. By the Implicit Function Theorem, there exists $\delta>0$ such that for each $\mu\in (\mu_3-\delta,\mu_3+\delta)$ there exist (unique) $(\lambda(\mu),\Omega(\mu),\Omega_1(\mu))$ such that $G(\mu,(\lambda(\mu),\Omega(\mu),\Omega_1(\mu)))=0_{\,\mathbb{R}^5}$, that is, the triple $(\lambda(\mu),\Omega(\mu),\Omega_1(\mu))$ satisfies \eqref{eq:eigenvalue_mu} with $\Omega(\mu)$ such that 
\begin{equation}\label{eq:constraint_mu}
\inner{\Omega(\mu)}{\mathsf{a}_3}=0,\qquad \norm{\Omega(\mu)}=1.
\end{equation}
Note that, again from the Implicit Function Theorem, $(\lambda(\mu),\Omega(\mu),\Omega_1(\mu))=g(\mu)$ where $g:\; (\mu_3-\delta,\mu_3+\delta)\to \mathbb{R}\times \mathbb{R}^3\times \mathbb{R}^3$ is a continuously differentiable function with $g(\mu_3)=(0,\mathsf{a}_2,\mathsf{a}_2)$. Adding side by side the equations in \eqref{eq:eigenvalue_mu}, we find that the vector field $K(\mu):=J_\mu\Omega(\mu)+I\Omega_1(\mu)$ satisfies the algebraic equation 
\[
-\alpha^*\mathsf{a}_3\times K(\mu)+\alpha^*\Lambda^*\mathsf{a}_3\times\Omega(\mu)=\lambda(\mu)K(\mu),
\]
where $\Lambda^*=\lambda^*+I\equiv\mu_3+I$. Let us differentiate both sides of the above equation with respect to $\mu$ and evaluate the resulting equation at $\mu_3$, we find 
\begin{equation}\label{eq:mu_III}
\alpha^*\mathsf{a}_1-\alpha^*\mathsf{a}_3\times(J_{\mu_3}-\mu_3\mathbf{E}_0)\dot \Omega(\mu_3)-\alpha^*\mathsf{a}_3\times I(\dot\Omega_1(\mu_3)-\dot\Omega(\mu_3))=\dot \lambda(\mu_3)\Lambda^*\mathsf{a}_2.
\end{equation}
In the above, the dot `` $\dot{}$ '' denotes the differentiation with respect to the parameter $\mu$. Note also that 
\[
(J_{\mu_3}-\mu_3\mathbf{E}_0)\dot \Omega(\mu_3)=(\mu_1-\mu_3)\inner{\dot\Omega(\mu_3)}{\mathsf{a}_1}\mathsf{a}_1,
\]
and recall that $\{\mathsf{a}_1,\mathsf{a}_2,\mathsf{a}_3\}$ are eigenvectors of $J$ chosen to form an orthonormal basis of $\mathbb{R}^3$. So, \eqref{eq:mu_III}  implies in particular that 
\begin{equation}\label{eq:mu_V}
\inner{\dot\Omega_1(\mu_3)-\dot\Omega(\mu_3)}{\mathsf{a}_2}=-I^{-1}.
\end{equation}
Now, let us differentiate both sides of \eqref{eq:eigenvalue_mu}$_2$ with respect to $\mu$ and evaluate the resulting equation at $\mu_3$, we obtain 
\begin{equation}\label{eq:mu_VI}
-k(\dot\Omega_1(\mu_3) - \dot\Omega(\mu_3)) -\alpha^*\,\mathsf{a}_3 \times I(\dot\Omega_1(\mu_3)-\dot\Omega(\mu_3))=\dot\lambda(\mu_3) I\mathsf{a}_2.
\end{equation}
Using \eqref{eq:mu_V}, it follows that 
\[
\inner{\dot\Omega_1(\mu_3) - \dot\Omega(\mu_3)}{\mathsf{a}_1} =-\frac{\alpha^*}{k},
\]
and then 
\[
\dot\lambda(\mu_3)=\frac{k}{I^2}+\frac{(\alpha^*)^2I}{k}>0. 
\]
Therefore, the eigenvalue $\lambda(\mu)$ crosses the imaginary axis with positive speed along the real axis when $\mu=\mu_3$. By a standard perturbation argument for simple eigenvalues (see, e.g., \cite[Theorem 3.2 \& Remark 3.4]{Crandall1973}), we also infer that $\lambda(\mu)$ exists for every $\mu>\mu_3$. In addition, $\lambda(\mu)>0$ for $\mu>\mu_3$ thanks to the first part of this proof. 

Now set $\mu_1=A_1$, $\mu_3=A_2$ and $\mu=A_3$. The above argument shows that the equilibrium $(\Omega^*,\Omega_1^*)$ with $\Omega^*=\Omega_1^*\in \mathsf{eigen}_J(\lambda^*\equiv A_2)$ is unstable since there is exactly one positive eigenvalue given by $\lambda(\mu=A_3)$. 

Now assume that  
\begin{equation}\label{eq:J_mu2}
J_\mu:=\left[\begin{matrix}
\mu & 0 & 0
\\
0 & \mu_2 & 0
\\
0 & 0 & \mu_3\end{matrix}\right],
\end{equation}
with $\mu_2>\mu_3$ and $\mu\in (\mu_3-\delta,\mu_3+\delta)$. Let us consider the equilibrium $(\Omega^*,\Omega_1^*)$ with $\Omega^*=\Omega_1^*=\alpha\mathsf{a}_3$. Thanks to what we have proved above, we can conclude that $L(\mu)$ has $\lambda=0$ as simple eigenvalue, one positive eigenvalue and all the other eigenvalues with negative real part as long as $\mu<\mu_3$. Note, in particular, that $\lambda=0$ becomes an eigenvalue of multiplicity two as soon as $\mu=\mu_3$. This shows that, also in the case $\mu=\mu_3=A_1=A_2<\mu_2=A_3$, there is exactly one positive eigenvalue of $L(\mu)$. We are now ready to prove property (F2), that is the existence of two eigenvalues of $L(\mu)$\footnote{This is the linear transformation defined in \eqref{eq:linearization} with $J$ replaced by $J_\mu$ in \eqref{eq:J_mu2}. } with positive real part whenever $\mu>\mu_3$.  In fact, we can redo the argument with the Implicit Function Theorem with $\mathsf{a}_2$ replaced by $\mathsf{a}_1$, since now the map $G$ (defined in \eqref{eq:implicit_G}) satisfies $G(\mu_3,(0,\mathsf{a}_1,\mathsf{a}_1))=0_{\,\mathbb{R}^5}$. Therefore, since $\mu_2>\mu_3$, there will be two eigenvalues with positive real part for any $\mu>\mu_3$. 

Take now $\lambda^*=\mu_3=A_1<\mu=A_2< \mu_2=A_3$. We have just showed that there will be two eigenvalues of $L(\mu=A_2)$ with positive real part. In addition $\lambda=0$ is a simple eigenvalue and all remaining eigenvalues have negative real part. Note that this assertion still holds true in the case $\lambda^*=\mu_3=A_1<\mu= \mu_2=A_2=A_3$. The proof is then complete since $\sigma(L^*)\equiv \sigma(L(\mu))$ for the above choices of $\mu$.
\end{proof}

Thanks to the above result and Theorems \ref{normalstable} \& \ref{normalhyper}, we immediately have the following result concerning the nonlinear stability properties of the equilibria of \eqref{ODE}. 

\begin{theorem}\label{th:nonlinear_stability}
Let $(\Omega^*,\Omega_1^*)$ be an equilibrium point of \eqref{ODE} such that $\Omega^* = \Omega_1^*$ is an eigenvector of $J$ corresponding to $\lambda^*\in \{A_1,A_2,A_3\}$, and recall that $A_1\le A_2\le A_3$. Then, the equilibrium  $(\Omega^*,\Omega_1^*)$ is exponentially stable\footnote{``Exponential stability'' should be read in the sense of Theorem \ref{normalstable}. } for the nonlinear equation \eqref{ODE} if $\lambda^*=A_3$. 

If $\lambda^*\in \{A_1,A_2\}$ with $A_1\le A_2<A_3$, then $(\Omega^*,\Omega_1^*)$ is unstable.
\end{theorem}

\section{Long-time behaviour of solutions}\label{sec:long-time}
\subsection{``Guaranteed'' convergence}
In this section, we extend the (nonlinear) stability results to conclude that \textbf{all} solutions to \eqref{ODE} converge to a permanent rotation about a principal axis of inertia of the outer body. Let $(\Omega(t), \Omega_1(t))$ denote a solution to \eqref{ODE} corresponding to the ``generic'' initial condition $(\Omega_0,\Omega_{10})\in \mathbb{R}^3\times\mathbb{R}^3$. Note that we are considering initial data not necessarily close to an equilibrium. 

We start by applying Lemma \ref{gronwall} to the balance of kinetic energy  \eqref{tderiv-kin} to obtain a decay for the relative velocity $\Omega -\Omega_1$.

\begin{theorem}
\label{gronwall-app}
Assume that $k > 0$. The relative velocity $\Omega(t) - \Omega_1(t)$ converges to $0$ as $t \rightarrow \infty$.
\end{theorem}
\begin{proof}
We start by defining the nonnegative function $y = -\dot{V} = k\norm{\Omega - \Omega_1}^2$. From \eqref{tderiv-kin}, we have that $\Omega$ and $\Omega_1$ are both uniformly bounded, and thus $y \in L^\infty(0,\infty)$. We want to get an inequality of the form:
\begin{equation*}
    \dot y(t) \leq -Cy(t) + F(t),\quad t \in [0,\infty),
\end{equation*}
where $C$ is a positive constant, and $F$ is a non-negative function in $L^q(0,\infty) \cap L^1_{loc}(0,\infty)$, for some $q\in [1,\infty)$. We directly compute the time-derivative of $y$ using \eqref{tderiv-kin} and \eqref{ODE} to obtain 
\begin{align*}
     \dot y &= - \ddot V 
     = - 2k^2\inner{\Omega - \Omega_1}{J^{-1}(\Omega - \Omega_1)} - 2k^2\inner{\Omega - \Omega_1}{I^{-1}(\Omega - \Omega_1)} 
     \\
     &\quad- 2k\inner{\Omega - \Omega_1}{J^{-1}(\Omega \times (J\Omega))}
     \\
     &\leq -2k^2\frac{I+A_3}{A_3 I}\norm{\Omega - \Omega_1}^2 + 2k\frac{A_3}{A_1}\norm{\Omega}^2\norm{\Omega - \Omega_1}
     \\
     &\le -2k^2\frac{I+A_3}{A_3 I}\norm{\Omega-\Omega_1}^2 +M\norm{\Omega - \Omega_1},
\end{align*}
for some positive constant $M=M(\Omega_0,\Omega_{10},J,k)$. In the above, we have used Cauchy-Schwarz inequality, the fact that $J$ is a diagonal tensor with eigenvalues $A_1\le A_2\le A_3$, and that $\norm{\Omega}$ is uniformly bounded. 

If we define $F(t) := M\norm{\Omega - \Omega_1}$ and $C :=\displaystyle 2k^2\frac{I+A_3}{A_3 I}> 0$, we have the desired inequality 
\begin{align*}
    \dot y(t) &\leq -Cy(t) + F(t). 
\end{align*}
It remains to show that $F$ satisfies the requirements of Lemma \ref{gronwall}. We indeed have that $F(t) \geq 0$, and since it is continuous, it is locally integrable. By the balance of kinetic energy \eqref{tderiv-kin} together with the Fundamental Theorem of Calculus, we have that 
\begin{align*}
    \int^{\infty}_{0}\norm{\Omega(\tau) - \Omega_1(\tau)}^2 d\tau &= \lim_{t \rightarrow \infty} \frac{V(0) - V(t)}{k} = \frac{V(0) - V_\infty}{k}< \infty.
\end{align*}
In the above equation, $V_\infty$ is defined as $\lim_{t \rightarrow \infty} V(t)$, which exists since the kinetic energy  $V(t)$ is a non-negative, non-increasing continuous function. Hence $F \in L^2(0,\infty)\cap L^1_{loc}(0,\infty)$. By Lemma \ref{gronwall},  we can finally conclude that $y(t)\to 0$ as $t\to\infty$, which in turn implies that 
\begin{gather*}
    \lim_{t \rightarrow \infty} (\Omega(t) - \Omega_1(t)) = 0.
\end{gather*}
\end{proof}
As a consequence of the above theorem, we then have the follow corollary. 
\begin{corollary}
\label{gronwall-coro}
Solutions $(\Omega,\Omega_1)$ of \eqref{ODE} satisfy 
\[
\lim_{t \rightarrow \infty} \dot{\Omega}(t) = \lim_{t \rightarrow \infty} \dot{\Omega}_1(t) = 0.
\]
\end{corollary}
\begin{proof}
Using the second equation of \eqref{ODE} and noting $\Omega \times \Omega = 0$,  we have that 
\begin{align*}
    \dot{\Omega}_1 
   &= \frac{k}{I}(\Omega - \Omega_1) + \Omega \times (\Omega - \Omega_1). 
\end{align*}
By the triangle inequality and the properties of the cross product, we have 
\begin{align*}
    \norm{\dot{\Omega}_1} &\leq \frac{k}{I}\norm{\Omega - \Omega_1} + \norm{\Omega \times (\Omega - \Omega_1)} 
    \leq \left(\frac{k}{I} + M\right) \norm{\Omega - \Omega_1},
\end{align*}
where $M \geq 0$ is a uniform bound on $\Omega$ (depending on the norm of the initial conditions) resulting from \eqref{tderiv-kin}. We then get 
\[
\lim_{t\rightarrow\infty} \norm{\dot{\Omega}_1(t)} \leq \left(\frac{k}{I} + M\right)\lim_{t\rightarrow\infty}  \norm{\Omega(t) - \Omega_1(t)} = 0,
\]
and so we have that $\lim_{t \rightarrow \infty} \dot{\Omega}_1(t) = 0$. We recognize that the vector field $f(t):=\dot{\Omega} (t)- \dot{\Omega}_1(t)$ is a uniformly continuous function, since the right-hand sides of \eqref{ODE} are Lipschitz continuous. In addition, thanks to the previous theorem, we know that there exists 
\[
\lim_{t\to\infty}\int^t_0f(\tau)\;d\tau=%\Omega(t)-\Omega_1(t)+\Omega_{10}-\Omega_0=
\Omega_{10}-\Omega_0.
\]
Thus, using Barb\u alat's Lemma for vector valued functions (see e.g., \cite{barb}), we get that there exists 
\[
\lim_{t \rightarrow \infty} f(t)=\lim_{t \rightarrow \infty} (\dot{\Omega}(t) - \dot{\Omega}_1(t)) = 0. 
\]
Therefore, there also exists
\[
\lim_{t\to\infty}\dot{\Omega}(t)=\lim_{t\to \infty}\dot{\Omega}_1(t)+\lim_{t\to\infty} (\dot{\Omega}(t) - \dot{\Omega}_1(t)) = 0.
\]
\end{proof}

We can now combine results from the previous section and this section to characterize the long-time behaviour of solutions to \eqref{ODE}.

\begin{theorem}\label{th:long-time}
All solutions to \eqref{ODE} must eventually converge to a permanent rotation about one of the principal axes of inertia of $J$ with an exponential rate. 

More precisely, for every $(\Omega_0,\Omega_{10})\in \mathbb{R}^3\times\mathbb{R}^3$, there exist $t^*\ge0$ and $(\Omega^*,\Omega_1^*)\in \mathcal E$ (see Theorem \ref{equilibria}) such that the (unique) solution $(\Omega (t),\Omega_1(t))$ of \eqref{ODE}, corresponding to the initial conditions $(\Omega_0,\Omega_{10})$, exists for all $t\ge 0$, and satisfies the estimate
\[
\norm{\Omega(t)-\Omega^*}+\norm{\Omega_1(t)-\Omega^*}\le Ce^{-ct}\qquad\text{for all }t\ge t^*,
\]
where $C$ and $c$ are positive constants. 
\end{theorem}
\begin{proof}
Let $z(t;z_0) := (\Omega(t), \Omega_1(t))$ denote the (unique) solution to \eqref{ODE} corresponding to the initial condition $z_0:=(\Omega_0,\Omega_{10})\in \mathbb{R}^3\times\mathbb{R}^3$. The positive orbit 
\[
\gamma^+(z_0):=\{(z_1,z_2)\in \mathbb{R}^3\times\mathbb{R}^3:\; (z_1,z_2)=z(t;z_0)\text{ for some }t\ge 0\}
\]
is bounded by \eqref{tderiv-kin}. So, the $\omega$-limit set 
\[\begin{split}
\omega(z_0):=\{(u,v)\in\mathbb{R}^3\times\mathbb{R}^3&:\; \text{there exists }\{t_n\}_{n=1}^\infty\text{ with } t_n\nearrow+\infty\text{ as }n\to\infty, 
\\
&\text{such that }\lim_{n\to\infty}\norm{\Omega(t_n)-u}=\lim_{n\to\infty}\norm{\Omega_1(t_n)-v}=0\}
\end{split}\]
is non-empty, compact and connected. In addition (see e.g., \cite[Chapter I, Theorem 8.1]{hale}), 
\begin{equation}\label{eq:distance}
\lim_{t \rightarrow \infty} \text{d}(z(t;z_0),\omega(z_0))=0. 
\end{equation}
In view of Theorem \ref{gronwall-app}, Corollary \ref{gronwall-coro}, and Theorem \ref{equilibria}, we have that $\omega(z_0)\subset\mathcal E$. Since $\omega(z_0) $ is connected, by Theorem \ref{pruss-theorem}, we then have two possibilities: either  $\omega(z_0) $ contains only normally stable equilibria or it contains only normally hyperbolic equilibria. 

Let $z^*\in\omega(z_0)$ and assume it is a normally stable equilibrium. Then, $z^*=(\Omega^*,\Omega^*_1)$ with $\Omega^*=\Omega^*_1=\alpha^*\mathsf{a}_3$ an eigenvector of $J$ corresponding to the principal moment of inertia $\lambda^*=A_3$ (see Theorems \ref{equilibria} and \ref{pruss-theorem}, and recall that -without loss of generality- the principal moment of inertia of $\mathcal B_1$ have been ordered as $A_1\le A_2\le A_3$). There exists a sequence $\{t_n\}_{n=1}^\infty\subset(0,\infty)$, $t_n\nearrow+\infty$ as $n\to\infty$, such that 
\[
\lim_{n\to\infty}\norm{\Omega(t_n)-\Omega^*}=\lim_{n\to\infty}\norm{\Omega_1(t_n)-\Omega^*}=0. 
\]
Corresponding to $\delta$ in Theorem \ref{normalstable}, there exists $N>0$ such 
$\norm{\Omega(t_N)-\Omega^*}+\norm{\Omega_1(t_N)-\Omega^*}<\delta$. Now, we consider $z_N:=(\Omega(t_N),\Omega_1(t_N))$ as initial condition of the solution $z(t;z_N)$ of \eqref{ODE}. By uniqueness of solutions, 
\[
z(t;z_N)=z(t;z(t_N;z_0))=z(t+t_N;z_0), 
\]
and again by Theorem \ref{normalstable}, there exists $z_\infty\in \mathcal E$ such that $\norm{z(t;z_N) - z_{\infty}} \leq c_1e^{-c_2t}$ for all $t \geq 0$, that is 
\[
\norm{z(t;z_0) - z_{\infty}} \leq c_1e^{-c_2(t-t_N)}\quad \text{for all } t \geq t_N.
\]

Now, suppose that $\omega(z_0)$ contains only normally hyperbolic equilibria. Fix $\varepsilon>0$.  By \eqref{eq:distance}, there exists $t^*>0$ such that 
\[
\text{d}(z(t;z_0),\mathcal E)\le \text{d}(z(t;z_0),\omega(z_0))<\varepsilon \quad\text{ for all }t\ge t^*.
\]
Consider $z^*=(\Omega^*,\Omega^*)\in \omega(z_0)$, and a sequence $\{t_n\}_{n=1}^\infty\subset(0,\infty)$, $t_n\nearrow+\infty$ as $n\to\infty$, such that 
\[
\lim_{n\to\infty}\norm{\Omega(t_n)-\Omega^*}=\lim_{n\to\infty}\norm{\Omega_1(t_n)-\Omega^*}=0. 
\]
Take $\delta$ from the second part of Theorem \ref{normalhyper} with $\rho=\varepsilon$. Corresponding to $\delta$, there exists $N>0$ such that $\norm{\Omega(t_n)-\Omega^*}+\norm{\Omega_1(t_n)-\Omega^*}<\delta$ for all $n>N$. Now, fix $m>N$ such that $t_m\ge t^*$, and consider $z_m:=(\Omega(t_m),\Omega_1(t_m))$ as initial condition of a (unique) global solution $z(t;z_m)$ of \eqref{ODE}. Such a solution satisfies 
\[
\text{d}(z(t;z_m),\mathcal E)=\text{d}(z(t+t_m;z_0),\mathcal E)\le\text{d}(z(t+t_m;z_0),\omega(z_0))<\varepsilon \quad\text{ for all }t\ge 0.
\]
Therefore, the second part of Theorem \ref{normalhyper} implies the existence of  $z_\infty\in \mathcal E$ such that $\norm{z(t;z_m) - z_{\infty}} \leq c_1e^{-c_2t}$ for all $t \geq 0$, that is 
\[
\norm{z(t;z_0) - z_{\infty}} \leq c_1e^{-c_2(t-t_m)}\quad \text{for all } t \geq t_m.
\]
\end{proof}

\subsection{Attainability of equilibria}
\label{sect-attain}

Theorem \ref{th:long-time} is silent about the equilibrium configurations that solutions to \eqref{ODE} will eventually attain. Physically speaking, after imparting an initial angular momentum, say to $\mathcal B_1$, the whole system $\mathcal B_1\cup\mathcal B_2$ will eventually reach the equilibrium which is a permanent rotation around one of the principal axes of inertia of $\mathcal B_1$. In other words, $\mathcal B_2$ will eventually go to the rest relative to $\mathcal B_1$, and the whole system $\mathcal B_1\cup\mathcal B_2$ will be spinning (with constant angular velocity) around one of the principal axes of inertial $\mathcal B_1$.  However, it is not known a priori (depending on the initial data) around which axis this rotation will occur. In view of Theorem \ref{normalstable}, we know that stable equilibria are only those rotations around the axis of inertia corresponding to the largest moment of inertia. Objective of this section is to obtain sufficient conditions on the initial conditions that would ensure that the whole system would eventually attain a permanent rotation around the principal axis corresponding to the largest moment of inertial $A_3$ of $\mathcal B_1$. 

To start, recall the definition of $K^2$ and $V$ (i.e., the squared norm of the angular momentum and the kinetic energy, respectively), 
\[\begin{split}
V(\Omega,\Omega_1) &= \frac{1}{2}(\inner{\Omega}{J\Omega} + I\inner{\Omega_1}{\Omega_1})
\\
K^2(\Omega,\Omega_1)&=\norm{J\Omega + I\Omega_1}^2
\end{split}\]
From Proposition \ref{prop:balances}, the following balances hold along solutions to \eqref{ODE}: 
\begin{gather}
\label{tderiv}
    \frac{d}{dt}K^2 = 0, \quad \frac{d}{dt}V = -k\norm{\Omega_1(t) - \Omega(t)}^2
\end{gather}
Let $(\Omega,\Omega_1)$ be a solution to \eqref{ODE}, and define 
\[\begin{split}
\bar\Omega:=\lim_{t\to\infty}\Omega(t)=\lim_{t\to\infty}\Omega_1(t). 
\end{split}\]
We note that the above limits exist by Theorem \ref{th:long-time}. 
Integrating both equations in \eqref{tderiv} over $[0,\infty)$, we find that $\bar\Omega$ must satisfy the following two conditions 
\begin{equation}
\label{integrated}
\begin{split}
    \norm{(J+I)\bar\Omega}^2 &= K^2(0), \quad 
    \\
    \frac{1}{2}\inner{\bar\Omega}{(J+I)\bar\Omega} &= V(0) - \int_{0}^{\infty}k\norm{\Omega_1(\tau) - \Omega(\tau)}^2d\tau. 
\end{split}
\end{equation}
Let us write $\Omega(t)=p(t)\mathsf{a}_1+q(t)\mathsf{a}_2+r(t)\mathsf{a}_3$, and denote  
\[\begin{split}
\bar p:=\lim_{t\to\infty}p(t),\quad \bar q:=\lim_{t\to\infty}q(t),\quad \bar r:=\lim_{t\to\infty}r(t).
\end{split}\]
Equations \eqref{integrated} can be rewritten as follows 
\begin{multline}
\label{momlim}
 (A_1 + I)^2\bar{p}^2 + (A_2 + I)^2\bar{q}^2 + (A_3 + I)^2\bar{r}^2 \\
 = A_1^2p^2(0) + A_2^2q^2(0) + A_3^2r^2(0) + I^2\norm{\Omega_1(0)}^2 + 2\inner{J\Omega(0)}{I\Omega_1(0)},
\end{multline}
and
\begin{multline}
\label{kinmom}
 (A_1 + I)\bar{p}^2 + (A_2 + I)\bar{q}^2 + (A_3 + I)\bar{r}^2 \\
 \qquad\ = A_1p^2(0) + A_2q^2(0) + A_3r^2(0) + I\norm{\Omega_1(0)}^2 -\int_{0}^{\infty}k\norm{\Omega_1(\tau) - \Omega(\tau)}^2d\tau. 
\end{multline}
We denote the latter displayed integral as $D_{\infty}$. We are ready to present the following result about the attainability of permanent rotations about the principal axis corresponding to the largest moment of inertia of $\mathcal B_1$. 
\begin{theorem}
\label{attainability} We have the following:
\begin{enumerate}
    \item[(a)] Suppose that $A_1 = A_2 < A_3$. If 
            \begin{gather*}
            (A_3 - I)(A_3 - A_1)r^2(0) + 2\inner{J\Omega(0)}{I\Omega_1(0)} > A_1I(\norm{\Omega(0)}^2 + \norm{\Omega_1(0)}^2),
        \end{gather*}
        then $$\lim_{t \rightarrow \infty} p(t) = \lim_{t \rightarrow \infty} q(t) = 0. $$
    \item[(b)] Suppose that $A_1 < A_2 < A_3$. If the following two conditions are satisfied
        \begin{equation*}\begin{split}
            (A_2 - I)(A_2 - A_1)q^2(0) +  (A_3 - I)(A_3 - A_1)&r^2(0) + 2\inner{J\Omega(0)}{I\Omega_1(0)} 
            \\
            &> A_1I(\norm{\Omega(0)}^2 + \norm{\Omega_1(0)}^2), 
            \\
            (A_1 - I)(A_1 - A_2)p^2(0) + (A_3 - I)(A_3 - A_2)&r^2(0) + 2\inner{J\Omega(0)}{I\Omega_1(0)} 
            \\
            &> A_2I(\norm{\Omega(0)}^2 + \norm{\Omega_1(0)}^2), 
        \end{split}\end{equation*}
        then $$\lim_{t \rightarrow \infty} p(t) = \lim_{t \rightarrow \infty} q(t) = 0. $$
    \item[(c)] Suppose that $A_1 < A_2 = A_3$. If 
        \begin{equation*}\begin{split}
            (A_3 - I)(A_3 - A_1)(q^2(0) + r^2(0)) + 2&\inner{J\Omega(0)}{I\Omega_1(0)} 
            \\
            &\qquad > A_1I(\norm{\Omega(0)}^2 + \norm{\Omega_1(0)}^2),
        \end{split}\end{equation*}
        then $$\lim_{t \rightarrow \infty} p(t) = 0.$$ 
\end{enumerate}
\end{theorem}
\begin{proof}
We start with statement (a). We argue by contradiction, and assume that $\bar{r} = 0$. By Theorem \ref{th:long-time}, we know that either $\bar{p} \neq 0$ or  $\bar{q} \neq 0$ (or both are nonzero). Equations \eqref{momlim} and \eqref{kinmom} then become
\begin{equation*}\begin{split}
     (A_1 + I)^2(\bar{p}^2 + \bar{q}^2) &= A_1^2p^2(0) + A_2^2q^2(0) + A_3^2r^2(0) + I^2\norm{\Omega_1(0)}^2 
     \\
     &\qquad\qquad\qquad\qquad\qquad\qquad\qquad\qquad\qquad\quad
     + 2\inner{J\Omega(0)}{I\Omega_1(0)}, 
     \\
     (A_1 + I)(\bar{p}^2 + \bar{q}^2) &= A_1p^2(0) + A_2q^2(0) + A_3r^2(0) + I\norm{\Omega_1(0)}^2 - D_{\infty}. 
\end{split}\end{equation*}

Multiplying both sides of the second equation with $(A_1 + I)$ and then subtracting the resulting equation from the first one, we obtain
\begin{equation*}%\label{nec_eq}
\begin{split}
    -(A_1 + I)D_{\infty} &= A_3^2r(0)^2 - A_1I(p^2(0) + q^2(0)) - (A_1 + I)A_3r^2(0) - A_1I\norm{\Omega_1(0)}^2 
    \\   &\qquad\qquad\qquad\qquad\qquad\qquad\qquad\qquad\qquad\qquad\quad
    + 2\inner{J\Omega(0)}{I\Omega_1(0)}
    \\
       &= (A_3 - I)(A_3 - A_1)r^2(0) - A_1I(\norm{\Omega(0)}^2 + \norm{\Omega_1(0)}^2) \\
       &\qquad\qquad\qquad\qquad\qquad\qquad\qquad\qquad\qquad\qquad\quad+ 2\inner{J\Omega(0)}{I\Omega_1(0)}. 
       \end{split}
\end{equation*}
Since $(A_1 + I)D_{\infty}$ must be nonnegative, we immediately find a contradiction with the hypothesis of statement (a). 

For statement (b), the two sufficient conditions are obtained in a similar manner by separately assuming $\bar{p} \neq 0$ and $\bar{q} \neq 0$, and following a similar procedure as with statement (a). Statement (c) follows as well by assuming $\bar{p} \neq 0$ and $A_2 = A_3$.
\end{proof}

In the next section, we will provide a numerical evidence that the conditions in the above theorem are only sufficient. In particular, we will show (numerically) that there exists trajectories corresponding to initial data not satisfying the condition in Theorem \ref{attainability}(a) that still converges to the principal axis corresponding to $A_3$ (see Figure \ref{fig-attain}). 

We would like to remark also that, with a similar argument to that yielding Theorem \ref{attainability}, one could also find sufficient conditions for convergence to a different principal axis. For example, a sufficient condition for convergence to a principal axis corresponding to $A_1 = A_2$ is:
\begin{multline*}
(A_1 - A_3)(A_1 - I)(\norm{\Omega(0)}^2 - r^2(0)) + 2\inner{J\Omega(0)}{I\Omega_1(0)} \\
- A_3I(\norm{\Omega(0)}^2 + \norm{\Omega_1(0)}^2)> 0. 
\end{multline*}

\section{Numerical experiments}\label{sec:numerics}
This section contains samples of numerical results obtained by solving \eqref{ODE} numerically (in Python) with different parameters and initial data. The first pair of plots, Figures \ref{fig-stable} and \ref{fig-unstable}, show how a small perturbation in the initial data can affect the long-time behaviour, and even convergence to an unstable equilibrium point. The parameters $k=I=1$, $A_1=A_2=3$ and $A_3=7$ were used together with the following initial conditions:
\begin{gather*}
    z_1 = \begin{bmatrix}1.5\\ 3\\ 0 \\ -1\\ -2\\ 0\end{bmatrix}, \quad
    z_2 = \begin{bmatrix}1.5\\ 3\\ 0 \\ -1\\ -2.01\\ 0\end{bmatrix}, \quad
\end{gather*}
respectively. In the initial condition $z_1$, $\Omega_0$ and $\Omega_{10}$ are scalar multiples of one another, and they are eigenvectors corresponding to $A_1 = A_2$, whereas the initial condition $z_2$ is the same as the first one except for subtracting $0.01$ from the fifth coordinate. The trajectories corresponding to each condition are shown below in Figure \ref{fig-stable} and Figure \ref{fig-unstable}, respectively. 
\begin{figure}[H]
\centering
\includegraphics[width=0.75\textwidth]{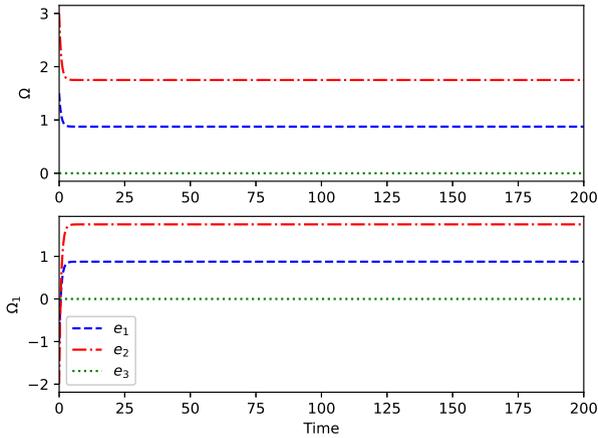}
\caption{Trajectory converges to a permanent rotation about the principal axis corresponding to $A_1 = A_2$.}
\label{fig-stable}
\end{figure}

\begin{figure}[H]
\centering
\includegraphics[width=0.75\textwidth]{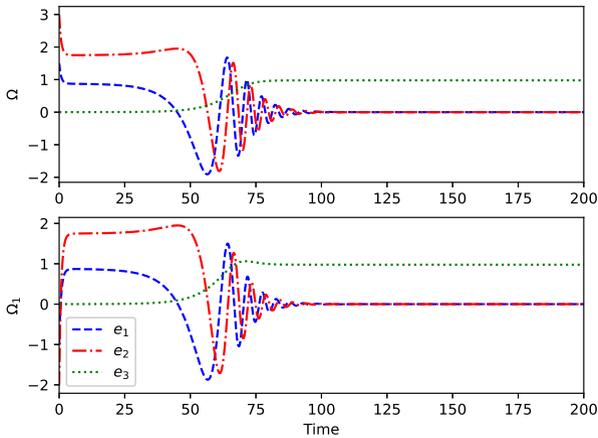}
\caption{Trajectory converges to a permanent rotation about the principal axis corresponding to $A_3$.}
\label{fig-unstable}
\end{figure}

Despite having ``close'' initial conditions, the trajectories from $z_1$ and $z_2$ converge to permanent rotations about different principal axes. 

Figure \ref{fig-attain} uses the same parameters with the initial condition $z = (1,0,0,0,1,0)^T$. One can verify that this initial data does not satisfy the conditions of Theorem \ref{attainability}(a); however, we can see that the solution still converges to the principal axis corresponding to the largest principal moment. % in Fig. \ref{fig-attain}:
\begin{figure}[H]
\centering
\includegraphics[width=0.75\textwidth]{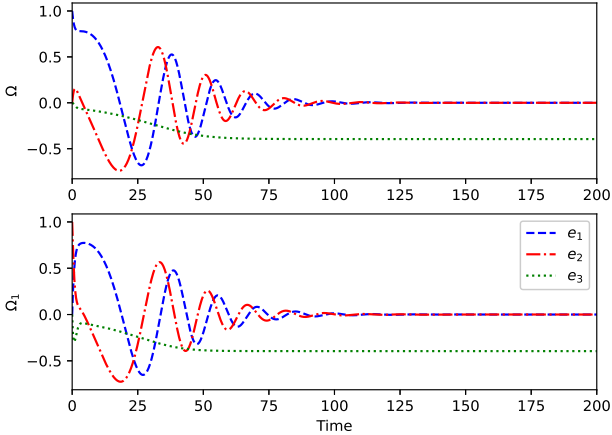}
\caption{Initial data does not satisfy Theorem \ref{attainability}(a), yet still converges to $\bar r \neq 0$. }
\label{fig-attain}
\end{figure}

\bibliography{Arsenault_Mazzone-Rigid_body_with_damper}% common bib file

\end{document}